\documentclass[reqno,english]{amsart}

\usepackage{amsmath,amsfonts,amssymb,graphicx,amsthm,enumerate,url}
\usepackage[noadjust]{cite}
\usepackage{stmaryrd}
\usepackage{comment,paralist}
\usepackage{mathrsfs,booktabs,tabularx}
\usepackage{xifthen,xcolor,tikz,setspace}
\usetikzlibrary{decorations.pathmorphing,patterns,shapes,calc,decorations}
\usetikzlibrary{decorations.pathreplacing}
\usepackage{mathtools}

\usepackage[colorinlistoftodos]{todonotes}
\usepackage[colorlinks=true]{hyperref}

\numberwithin{equation}{section}

\newcommand{\bin}{\operatorname{Bin}}

\newcommand{\ber}{\operatorname{Be}}
\renewcommand{\restriction}{\mathord{\upharpoonright}}
\renewcommand{\epsilon}{\varepsilon}
\newcommand{\given}{\;\big|\;}
\newcommand{\one}{\mathbf{1}}

\usepackage{color}

\newtheorem{maintheorem}{Theorem}

\newtheorem{theorem}{Theorem}[section]
\newtheorem*{theorem*}{Theorem}
\newtheorem{lemma}[theorem]{Lemma}

\newtheorem{proposition}[theorem]{Proposition}

\newtheorem{observation}[theorem]{Observation}
\newtheorem*{observation*}{Observation}
\newtheorem{fact}[theorem]{Fact}
\newtheorem{corollary}[theorem]{Corollary}

\newtheorem{remark}[theorem]{Remark}

\theoremstyle{definition}{

\newtheorem*{definition*}{Definition}

}

\newcommand{\E}{\mathbb E}
\newcommand{\N}{\mathbb N}
\renewcommand{\P}{\mathbb P}

\newcommand{\R}{\mathbb R}

\newcommand{\cA}{\mathcal A}
\newcommand{\cB}{\mathcal B}
\newcommand{\cC}{\mathcal C}

\newcommand{\cE}{\mathcal E}
\newcommand{\cF}{\mathcal F}
\newcommand{\cG}{\mathcal G}
\newcommand{\cK}{\mathcal K}
\newcommand{\cR}{\mathcal R}
\newcommand{\cT}{\mathcal T}
\newcommand{\cW}{\mathcal W}
\newcommand{\bH}{\mathbf H}
\newcommand{\bM}{\mathbf M}
\newcommand{\bX}{\mathbf X}

\newcommand{\bxi}{\boldsymbol \xi}
\newcommand{\fG}{\mathfrak G}
\newcommand{\fM}{\mathfrak M}
\newcommand{\sM}{\mathsf M}
\newcommand{\scM}{\mathscr M}
\newcommand{\scW}{\mathscr W}

\DeclareMathOperator{\var}{Var}
\DeclareMathOperator{\dist}{dist}
\DeclareMathOperator{\diam}{diam}
\DeclareMathOperator{\cov}{Cov}
\DeclareMathOperator{\wt}{w}
\DeclareMathOperator{\hgt}{ht}
\newcommand{\GHP}{\textsc{ghp}}

\newcommand{\selltwo}{\ell^2_\shortdownarrow}

\begin{document}

\title{Noise sensitivity of critical random graphs}

\author{Eyal Lubetzky}
\address{E.\ Lubetzky\hfill\break
Courant Institute\\ New York University\\
251 Mercer Street\\ New York, NY 10012.}
\email{eyal@courant.nyu.edu}

\author{Yuval Peled}
\address{Y.\ Peled\hfill\break
Courant Institute\\ New York University\\
251 Mercer Street\\ New York, NY 10012.}
\email{yuval.peled@courant.nyu.edu}

\begin{abstract}
We study noise sensitivity of properties of the largest components $(\cC_j)_{j\geq 1}$ of the random graph $\cG(n,p)$ in its critical window $p=(1+\lambda n^{-1/3})/n$. For instance, is the property ``$|\cC_1|$  exceeds its median size'' noise sensitive? 
Roberts and \c{S}eng\"{u}l (2018) proved that the answer to this is yes if the noise $\epsilon$ is such that $\epsilon \gg n^{-1/6}$, and conjectured the correct threshold is~$\epsilon \gg n^{-1/3}$. 
That is, the threshold for sensitivity should coincide with the critical window---as shown for the existence of  long cycles by the first author and Steif~(2015).

We prove that for $\epsilon\gg n^{-1/3}$ the pair of  vectors $ n^{-2/3}(|\cC_j|)_{j\geq 1}$ before and after the~noise converges in distribution to a pair of  \emph{i.i.d.} random variables, whereas for $\epsilon\ll n^{-1/3}$ the $\ell^2$-distance between the two goes to 0 in probability. This confirms the above conjecture: any Boolean function of the vector of rescaled component sizes is sensitive in the former case and stable in the latter.

We also look at the effect of the noise  on the metric space $n^{-1/3}(\cC_j)_{j\geq 1}$. E.g., for  $\epsilon\geq n^{-1/3+o(1)}$,  we show that the joint law of the spaces before and after the noise converges to a product measure, 
implying noise sensitivity of any property seen in the limit, e.g., ``the diameter of $\cC_1$ exceeds its median.'' 
\end{abstract}

 \maketitle
{\mbox{}\vspace{-0.6cm}
\maketitle
}
\vspace{-0.4in}

\section{Introduction} 
A sequence of Boolean functions $f_N$ on a product space of $N$ Bernoulli variables is \emph{noise sensitive}, a notion introduced in the pioneering paper of Benjamini, Kalai and Schramm~\cite{BKS99}, if the resampling of an $\epsilon$-fraction of the inputs, for a fixed $\epsilon>0$, results in an asymptotically independent copy of $f_N$ in the limit as $N\to\infty$.

More precisely, consider i.i.d.\ Bernoulli($p_N$) random variables $\{\omega(i)\}_{i=1}^N$ and  functions $f_N:\{0,1\}^N\to\{0,1\}$, such that $\P(f_N(\omega)=1)$ is uniformly bounded away from $0$ and $1$. Let $\omega^\epsilon$ be the result of applying the noise operator $T_\epsilon$ on~$\omega$, whereby every $\omega_i$ is resampled (replaced by a new i.i.d.\ Bernoulli($p$) variable) with probability $\epsilon$, independently of the other coordinates.
The functions $f_N$ are said to be $\epsilon_N$-noise sensitive for $\epsilon=\epsilon_N>0$ if $\cov(f_N(\omega),f_N(\omega^\epsilon))\to 0$ as $N\to\infty$.

A remarkable argument of~\cite{BKS99} (see~\cite[\S1]{GS15}) showed that if $p_N\equiv\frac12$ and $f_N$ is increasing w.r.t.\ the partial order on the hypercube $\{0,1\}^N$, then noise sensitivity w.r.t.\ any (and every) $0<\epsilon<1$ occurs if and only if $\lim_{N\to\infty}\sum_{i=1}^N \hat f_N(\{i\})^2 = 0$, where $\hat f_N(\{i\})$ is the Fourier coefficient of $f_N$ corresponding to the singleton $i$. 
This ceases to apply when $p_N\to 0$ fast with $N$, as Benjamini et al.~\cite[\S6.4]{BKS99} noted: ``\emph{When $p$ tends to zero with $n$, new phenomena occur. Consider, for example, random graphs on $n$ vertices with edge probability $p = n^{-a}$}.''

Consider the Erd\H{o}s--R\'enyi random graph $\cG(n,p)$---the undirected graph on $n$ vertices where each edge is independently present with probability $p$---in the critical window of its celebrated phase transition, $p=(1+\lambda n^{-1/3})/n$ for fixed $\lambda\in\R$.
Aldous~\cite{Aldous97} famously found
the limit law of its vector $\bX_n$ of rescaled component~sizes.
Here we show a dichotomy for noise sensitivity of functions in the critical window: if the noise parameter $\epsilon$ is such that $\epsilon \gg n^{-1/3}$ then 
any non-degenerate function of $\bX_n$ is noise sensitive, whereas if $\epsilon = o(n^{-1/3})$ it is noise stable (an opposite notion).
Analogous results are obtained for functions of the (rescaled) metric space of $G_n$.

Note that the threshold $\epsilon\gg n^{-1/3}$ to achieve noise sensitivity exactly coincides with the width of the critical window ($p=(1+O(n^{-1/3})p_c$ where $p_c=1/n$). 
As part of their study of noise sensitivity in $\cG(n,p)$, the first author and Steif~\cite{LS15} demonstrated this for the property ``containing a cycle of length $\ell\in(an^{1/3},bn^{1/3})$,'' and 
gave a heuristic explanation of why it should occur for a ``global'' function~$f_n$:
Viewing the operator $T_\epsilon$ as a composition of sprinkling and percolation, the latter induces a subcritical graph where $f_n$ is degenerate, thus decorrelated (see~\cite[\S1.2]{LS15}).

Roberts and \c{S}eng\"{u}l~\cite{RS18} recently established that, for the largest component $\cC_1$ of~$\cG(n,p)$ at $p=1/n$, the property that $|\cC_1| \geq a n^{2/3}$ ($a>0$ fixed) is noise sensitive (N.B.\ the median of $|\cC_1|$ has order~$n^{2/3}$ at criticality). More precisely, they showed that this property is noise sensitive whenever $\epsilon \gg n^{-1/6}$,
and conjectured that the correct threshold for noise sensitivity is $\epsilon\gg n^{-1/3}$, in line with the above intuition.

Let $\cC_i=\cC_i(G)$ be the $i$-th largest component of a graph $G$, and write
\[ {\bf X}(G) = |V(G)|^{-2/3}(|\cC_1(G)|,|\cC_2(G)|,\ldots)\,,\]
viewing it as an element of $\selltwo$, the space of sequences $(x_i)_{i\geq 1}$ wherein the  $x_i\in\R_+$ are in decreasing order and $\sum x_i^2 < \infty$. Aldous~\cite{Aldous97} showed that when $G_n\sim\cG(n,p)$ for $p=(1+\lambda n^{-1/3})/n$ with $\lambda\in \R$ fixed, $\bX(G_n)\xrightarrow{\,\mathrm{d}\,}\bxi(\lambda)$ in the $\ell^2$-topology, where $\bxi(\lambda)$ is the sequence of excursion sizes, in decreasing order, of Brownian motion reflected at 0 with a parabolic drift as given by $\widehat W_t^{(\lambda)} = W^{(\lambda)}_t-\min_{0\leq s\leq t}W^{(\lambda)}_s$ where $W_t^{(\lambda)}= B_t+\lambda t - t^2/2$ for standard Brownian motion $B_t$.

Our first result confirms the above conjecture of Roberts and \c{S}eng\"{u}l, showing moreover that the ``noise sensitivity exponent'' (dubbed  so in~\cite[\S12]{GS15}) is indeed~$1/3$, coinciding with the width of the critical window, for any Boolean function of $\bX(G_n)$. 

\begin{maintheorem}\label{thm:1}
 Let $G_n\sim\cG(n,p)$ with $p=(1+\lambda n^{-1/3})/n$ for fixed $\lambda\in\R$, and let $G_n^\epsilon = T_\epsilon(G_n)$, where $T_\epsilon$ is the (Bonami--Beckner) noise operator for $0<\epsilon<1$.
 
 \begin{enumerate}[(i)]
     \item \label{it:sens} If $\epsilon=\epsilon(n)$ satisfies $\epsilon^3 n\to\infty$ as $n\to\infty$ then $(\bX(G_n),\bX(G_n^{\epsilon}))$ converges in distribution in $(\selltwo,\|\cdot\|_2)$ to a pair of independent copies of $\bxi(\lambda)$.
     \item \label{it:stab} If $\epsilon =\epsilon(n)$ satisfies $\epsilon^3 n \to 0$ as $n\to\infty$ 
     then
     $\|\bX(G_n)-\bX(G_n^{\epsilon})\|_2
     \xrightarrow{\,\mathrm{p}\,} 0$.
 \end{enumerate}
\end{maintheorem}
Part~\eqref{it:sens} implies that, for $\epsilon\gg n^{-1/3}$, any function of $\bX(G_n)$ is noise sensitive, whereas Part~\eqref{it:stab} implies that, for $\epsilon\ll n^{-1/3}$, any such function is noise stable in the sense that $\P(f_n(G_n)\neq f_n(G_n^\epsilon))\to 0$.
The joint law of $\bX(G_n),\bX(G_n^\epsilon)$ in the case where $\epsilon = t n^{-1/3}$ for a fixed $t> 0$ was recently obtained by Rossignol~\cite{Rossignol17} (a special case of the scaling limit of the pair of metric spaces, as mentioned below), whence the limit is nontrivial---neither two i.i.d.\ variables nor two identical ones.

The key to the proof of Part~\eqref{it:sens} is a short argument giving rigor to the  heuristic that  $f_n(G_n),f_n(G_n^\epsilon)$ decorrelate iff the noise $\epsilon$ ``draws'' $G_n$ outside of the critical window, where $f_n$ is degenerate.
 Suppose $\P(f_n(\omega)=1)\to \alpha$ for some $0<\alpha<1$.
By constructing $\omega$ via sprinkling on a subcritical $\omega_0$ comprised of i.i.d.\ $\ber(p_0)$ variables for $p_0=(1-\epsilon)p/(1-\epsilon p)$ (the percolation portion of the noise), one finds (see~Cor.~\ref{cor:common-core} and~\ref{cor:common-core-full}) that $f_n$ is noise sensitive if and only if $\P(f_n(\omega)=1 \mid \omega_0)\xrightarrow{\,\mathrm p\,}\alpha$. This reduces our problem to finding a set of graphs $\mathcal{H}_n$ such that $\P(H_n\in\mathcal{H}_n)\to 1$ for the subcritical $H_n\sim\cG(n,p_0)$, and for every $H_n\in\mathcal{H}_n$, the conditional law of $\bX(G_n)$ given $H_n$ converges to $\bxi(\lambda)$ in distribution. We then utilize a proposition of Aldous---driving his convergence result in~\cite{Aldous97}---which states that if the initial state of the multiplicative coalescent satisfies certain conditions, then its limit law is that of~$\bxi(\lambda)$. Verifying these conditions for almost every $H_n$ concludes Theorem~\ref{thm:1}\eqref{it:sens}.

We also study the exponent in the noise sensitivity threshold for properties of a critical $\cG(n,p)$ graph that are read, instead of from $\bX(G_n)$, rather from the rescaled metric spaces arising from the connected components of $G_n$. For a graph $G$, let
\[ \bM(G) = (\sM_1(G),\sM_2(G),\ldots)\,,\]
with $\sM_j(G)$ the measured metric space on $\cC_j(G)$ equipped with the graph metric scaled by $|V(G)|^{-1/3}$ and counting measure on vertices multiplied by $|V(G)|^{-2/3}$. A breakthrough result of Addario-Berry, Broutin and Goldschmidt~\cite{ABG12} established the scaling limit of the rescaled metric spaces in Gromov--Hausdorff distance ($d_{\textsc{gh}}$). This was extended to convergence in Gromov--Hausdorff-Prokhorov distance ($d_\GHP$) by those authors and Miermont in~\cite[Sec.~4.1]{ABGM17}, and we let $\fM(\lambda) = (\fM_j(\lambda):j\ge 1)$ denote the sequence of random measured metric spaces that describe the limit, where every $\fM_j$ is an element of $\scM$, the space of isometry-equivalence classes of compact measured metric spaces. Further let $d_\GHP^p(\bM,\bM') = (\sum_i d_\GHP(\sM_i,\sM_i')^p)^{\frac1p}$ and set
$\mathbb{L}_p=\{ \bM\in\scM^\N: \sum_i d_\GHP(\sM_i,\mathsf{Z})^p<\infty\}$ where $\mathsf{Z}$ is the zero metric space.

With this notation, our second result shows that the noise sensitivity exponent is $1/3$ for every Boolean function of $\bM(G_n)$ when $G_n$ is a critical $\cG(n,p)$ graph.
\begin{maintheorem}
\label{thm:2} Let $G_n\sim\cG(n,p)$ with $p=(1+\lambda n^{-1/3})/n$ for fixed $\lambda\in\R$, and let $G_n^\epsilon = T_\epsilon(G_n)$, where $T_\epsilon$ is the (Bonami--Beckner) noise operator for $0<\epsilon<1$.
 \begin{enumerate}[(i)]
     \item \label{it:sens2} Fix $\delta>0$. If $\epsilon(n) \geq n^{-1/3+\delta}$ then, for every $j$, the pair $(\sM_j(G_n),\sM_j(G_n^{\epsilon}))$ converges in distribution to a pair of independent copies of~$\fM_j(\lambda)$ in $(\scM\!,d_\GHP)$.
     
     Moreover, the pair
     $(\bM(G_n),\bM(G_n^{\epsilon}))$ converges in distribution to a pair of independent copies of $\fM(\lambda)$ in $(\mathbb L_4,d^4_{\GHP})$.
     \item \label{it:stab2} If $\epsilon(n) = o(n^{-1/3})$ 
     then
     $d_{\GHP}^4(\bM(G_n), \bM(G_n^{\epsilon}))
     \xrightarrow{\,\mathrm{p}\,} 0$ in $\mathbb L_4$.
 \end{enumerate}
\end{maintheorem}

As a consequence of Part~\eqref{it:sens2} of this theorem, in the regime
$\epsilon\geq n^{-1/3+\delta}$ we have that any nondegenerate function of $\bM(G_n)$ is noise sensitive. For example, this includes properties such as ``the diameter of $\cC_1(G_n)$ exceeds $a n^{1/3}$,'' (or its special case asking if this diameter exceeds its median), as well as the property ``$G_n$ contains a cycle of length between $a n^{1/3}$ and $b n^{1/3}$'' whose noise sensitivity was established in~\cite{LS15}. As mentioned there, the heuristic for the $1/3$ noise sensitivity exponent may fail for a ``local'' property (see Remark~3.5 in that work). Theorem~\ref{thm:2} thus shows that properties seen in the scaling limit $\fM(\lambda)$ qualify as ``global'' to that end.

The proof of Theorem~\ref{thm:2}\eqref{it:sens2} uses the same approach described above: one wishes to show that when constructing $G_n$ via sprinkling edges on a suitable $H_n$, the conditional law of $\bM(G_n)$ given $H_n$ weakly converges to $\fM(\lambda)$ for almost every $H_n$. Instead of the key proposition of Aldous on the convergence of the multiplicative coalescent from an initial state satisfying certain conditions, we employ an analog of it for metric spaces by Bhamidi, Broutin, Sen and Wang~\cite{BBSW14}.
To apply this powerful general result, of the various conditions that $H_n$ should satisfy, one must control the mean and variance of the average distance between all pairs of connected vertices.

As mentioned in the context of Theorem~\ref{thm:1}, the joint law of $(\bM(G_n),\bM(G_n^\epsilon))$ in the case where $\epsilon = t n^{-1/3}$ for fixed $t>0$ was recently derived by Rossignol~\cite{Rossignol17}, as a special case of a convergence result of $\{\bM(G_n(t))\}_{t\geq 0}$ where $\{G_n(t)\}_{t\geq 0}$ is the dynamical percolation process on the critical $\cG(n,p)$ graph with intensity $n^{-1/3}$. In principal, one could build on this to recover Part~\eqref{it:stab2} of Theorems~\ref{thm:1} and~\ref{thm:2} (which corresponds to~$t=o(1)$), yet given that the analysis of the case $t=o(1)$ is considerably simpler, we include the short proofs of stability for completeness.

This paper is organized as follows. Section~\ref{sec:common-core} describes the short
 reduction of noise sensitivity to controlling the conditional variance of $f_n(\omega)$ given a subcritical~$\omega_0$. Section~\ref{sec:subcritical} includes properties of the subcritical graph $\cG(n,p_0)$ which are needed for the proofs. Sections~\ref{sec:conv-X} and~\ref{sec:conv-M} then establish Theorems~\ref{thm:1} and~\ref{thm:2}, resp., via these tools.

\section{Conditional variance given a common core}\label{sec:common-core} 

For some sequences $0<p(n)<1$, $0<\epsilon(n)<1$ and $N(n)$, where $N$ is an integer going to $\infty$ with $n$, let $\Omega_n = \{0,1\}^{N}$ and let $\omega\sim\ber(p)^{\otimes N}$, the random variable taking values in $\Omega_n$ whose $N$ coordinates are i.i.d.\ $\ber(p)$ random variables.

For a given $\omega$, let $\omega^\epsilon = T_\epsilon\omega$, where $T_\epsilon$ is the (Bonami--Beckner) noise operator: $\omega^\epsilon(i) = \omega(i)$ with probability $1-\epsilon$ whereas $\omega(i)$ is a new independent $\ber(p)$ variable with probability $\epsilon$, independently between the $N$ coordinates. 

Next, define $\omega_0$, $\omega_1$ and $\omega_1'$ as follows. Let
\begin{equation}\label{eq:p0}
p_0 := \frac{p(1-\epsilon)}{1-\epsilon p}\,,
\end{equation}
and let $\omega_0\sim\ber(p_0)^{\otimes N}$. Conditional on $\omega_0$, let $\omega_1$ and $\omega_1'$ be independent outcomes of sprinkling on $\omega_0$:
\begin{equation}\label{eq:omega}
\omega_1(i)=\begin{cases} 1 & \omega_0(i)=1\\
\ber(p_1) & \omega_0(i)=0
\end{cases}
\qquad\mbox{where}\qquad
p_1 := \epsilon p\,,
\end{equation}
independently among the $N-|\omega_0|$ coordinates where $\omega_0(i)=0$, and similarly---via another set of $N-|\omega_0|$ independent $\ber(p_1)$ variables---for $\omega_1'$. 
\begin{observation}\label{obs:coupling}
For any $0<\epsilon<1$ and $0<p<1$
we have $(\omega_1,\omega_1')\stackrel{\mathrm{d}}= (\omega_1,\omega_1^\epsilon)$.
\end{observation}
\begin{proof}
Note that $\omega_1$  is distributed as $\ber(p)^{\otimes N}$ (as is $\omega_1'$) by our choice of $p_0,p_1$, as the $N$ variables are independent and $\P_n(\omega_1(i)=1) = p_0 + (1-p_0)p_1 = p$. Hence, such is also the distribution of $\omega_1^\epsilon$ (invariant under the noise operator). It thus remains to verify that the covariances of a given coordinate match. Indeed, for every~$i$,
\[ \P_n( \omega_1(i) = \omega_1'(i)=1) = p_0 + (1-p_0)p_1^2 =
p\frac{1-\epsilon+(1-p)\epsilon^2 p}{1-\epsilon p}=
p(1-\epsilon+\epsilon p)\,,\]
whereas
\[ \P_n( \omega_1(i) = \omega_1^\epsilon(i)=1) = p (1-\epsilon + \epsilon p)\,,\]
implying that the covariances indeed match, as required.
\end{proof}

\begin{corollary}\label{cor:common-core}
Let $f_n:\Omega_n\to\{0,1\}$. Then
\begin{equation}\label{eq:cond-var} \cov(f_n(\omega_1),f_n(\omega_1^\epsilon)) = \var\left(\P(f_n(\omega_1)=1\mid\omega_0)\right)\,.\end{equation}
In particular, if $\P_n(f_n(\omega_1)=1)\to\alpha$ as $n\to\infty$ for some $0<\alpha<1$, then the sequence
$(f_n)$ is $\epsilon$-noise sensitive if and only if $\P_n(f_n(\omega_1)=1\mid\omega_0)\xrightarrow{\,\mathrm{p}\,} \alpha$.
\end{corollary}
\begin{proof}
Eq.~\eqref{eq:cond-var} follows from  Observation~\ref{obs:coupling}, as the term $\cov(f_n(\omega_1), f_n(\omega_1') \mid \omega_0)$ in the law of total covariance for $\cov(f_n(\omega_1),f_n(\omega'_1))$ vanishes by the conditional independence of $\omega_1$ and $\omega_1'$ given $\omega_0$. Setting $X_n=\P(f_n(\omega_1)=1\mid \omega_0)$ for brevity, it thus follows by definition that if the (uniformly bounded) $X_n$ satisfies $X_n\xrightarrow{\,\mathrm{p}\,} \alpha$ then $\var(X_n)\to 0$, whence $(f_n)$ is $\epsilon$-noise sensitive. 

Conversely, if $(f_n)$ is $\epsilon$-noise sensitive then $\var(X_n)\to0$, so by Chebyshev's inequality 
$X_n - \E X_n \xrightarrow{\,\mathrm{p}\,}0 $, where $\E X_n = \P(f_n(\omega_1)=1)\to\alpha$ by assumption.
\end{proof}

If we further suppose that $X_n$ ($n\geq 1$) is a measurable map from $\Omega_n$ to a common topological space $\mathbb{S}$, which converges weakly to some limit $X_\infty$---denoting the measure $\P(X_\infty\in\cdot)$ on $\mathbb S$
is denoted $\mu_\infty$ for brevity---then we have the following.
\begin{corollary}
\label{cor:common-core-full}
Suppose that for each $n$ there exists $W_n\subseteq\Omega_n$ with $\P_n(W_n)\to 1$, the variables $X_n:\Omega_n\to\mathbb{S}$ have $X_n\xrightarrow{\,\mathrm d\,} X_\infty$, and for every $\mu_\infty$-continuity set~$S\subseteq\mathbb S$,
\begin{align}
     \lim_{n\to\infty} \E \left[\Big| \P_n( X_n(\omega_1) \in S \mid \omega_0) - \P(X_\infty\in S)\Big|^2\one_{\{\omega_0\in W_n\}}\right]
     = 0\,.
     \label{eq:common-core-full-hyp}
\end{align} 
Then $(X_n(\omega_1),X_n(\omega_1^\epsilon))\xrightarrow{\,\mathrm{d}\,}(X_\infty,X_\infty')$ as $n\to\infty$, where $X_\infty$ and $X'_\infty$ are i.i.d.
\end{corollary}
\begin{proof}
Let $f_n(\omega) = \one_{\{X_n(\omega)\in S_1\}}$ and $g_n(\omega) = \one_{\{X_n(\omega) \in S_2\}}$ for two $\mu_\infty$-continuity sets $S_1,S_2\subseteq\mathbb S$, and let $\alpha_1 = \P(X_\infty \in S_1)$. By the same argument that led to~\eqref{eq:cond-var},
\[ \cov(f_n(\omega_1),g_n(\omega_1^\epsilon)) = \cov\left(\P(f_n(\omega_1)=1\mid\omega_0)\,,\,\P(g_n(\omega_1)=1\mid\omega_0)\right)\,,\]
which vanishes as $n\to\infty$ if $Y_n := \P(f_n(\omega_1)=1\mid \omega_0)$ has $\var(Y_n)\to 0$. Since
\[ \left| \var(Y_n) - \E\big[|Y_n - \alpha_1|^2\one_{\{\omega_0\in W_n\}}\big]\right| \leq 3|\E Y_n - \alpha_1| + \P_n(\omega_0\notin W_n)\] 
for the random variables $0\leq Y_n \leq 1$, and each term on the right vanishes as~$n\to\infty$ by the hypotheses $X_n\xrightarrow{\,\mathrm d\,} X_\infty$ and $\P_n(W_n)\to 1$, resp., the proof is concluded.
\end{proof}

\begin{remark}\label{rem:common-core-weaker} When applying Corollary~\ref{cor:common-core-full}, we will establish the stronger inequality
\[
     \lim_{n\to\infty} \max_{\omega\in W_n} \left| \P_n( X_n(\omega_1) \in S \mid \omega_0=\omega) - \P(X_\infty\in S)\right| = 0\quad\forall\text{$\mu_\infty$-\emph{cont.\ set} } S\subseteq\mathbb S\,.
\]
Compare this with the definition of \emph{1-strong noise sensitivity} of $f_n:\Omega_n\to\{0,1\}$: a 1-\emph{witness} to $f_n$ is a minimal subset $W$ such that $\omega\restriction_W = 1$ implies $f_n(\omega)=1$ deterministically; if $\cW_n$ is the set of all witnesses, strong noise sensitivity says that  
\[
     \lim_{n\to\infty} \max_{W\in \cW_n} \left| \P_n( f_n(\omega^\epsilon) =1 \mid \omega\restriction_{W}=1) - \P_n(f_n=1)\right| = 0\,.
\]
This stronger notion implies, and is not equivalent to, noise sensitivity (see~\cite{LS15}).
Observe that strong noise sensitivity addresses the effect of conditioning that a specific subset $W$ of the variables is open in a critical configuration $\omega$. On the other hand,  here we examine the bias due to the entire configuration of the subcritical $\omega_0$.
\end{remark}
\section{Component structure of the subcritical random graph}\label{sec:subcritical}
Let $p=(1+\lambda n^{-1/3})/n$ for a fixed $\lambda \in \R$,
and $\epsilon=\epsilon_n\in(0,1)$ such that $\epsilon^3 n\to\infty$.
Letting $G\sim\cG(n,p)$ be the analog of $\omega_1$ in
in~\S\ref{sec:common-core} for the $N=\binom{n}2$ edge variables, define $G_0$ analogously to $\omega_0$; that is, $G_0\sim \cG(n,p_0)$ where $p_0$, as per~\eqref{eq:p0}, is set to
\begin{align}
\label{eqn:p_0_spelled_out}
p_0 &:=\frac{(1-\epsilon)(1+\lambda n^{-1/3})}{n-\epsilon(1+\lambda n^{-1/3})} = 
\frac{(1-\epsilon)(1+\lambda n^{-1/3})}n +O(n^{-2})\,.
\end{align}
It will be convenient to further denote
\begin{align}\label{eq:theta}
\theta &:= 1 - np_0= (1-o(1))\epsilon\,,\qquad\mbox{ whence }\theta^3 n\to\infty\,.
\end{align}
Our proofs will rely on several concentration estimates for features  of the subcritical random graph $G_0$, which are known to determine the scaling limit of the component sizes, and moreover, that of the full metric spaces of connected components of $G_1$ conditioned on $G_0$. 
In particular, we will need the following results.
\begin{theorem}[{\cite[Thms.~3.3,3.4,4.1]{JL08}}]
\label{thm:Jan_Luc_susceptibility}
For $G_0$ be as above and $r=2$ or $r=3$,
\begin{align}\label{eq:E-C^r}
\E\Big[\sum_{i}|\cC_i(G_0)|^r\Big] &= \left(1+O\left(\tfrac{1-\theta}{\theta^3 n}\right)\right) \theta^{3-2r} n\,,\\
\label{eq:Var-C^r}
\var\Big(\sum_{i}|\cC_i(G_0)|^r\Big) &= O\left( \theta^{3-4r} n\right)\,.
\end{align}
\end{theorem}

\begin{theorem}[{\cite{Bollobas84,Luczak90}}; {\cite[Cor.~5.11]{Bollobas01},\cite[Thm.~5.6]{JLR00}}]\label{thm:Bol-Luc-C1}
For $G_0$ as above,
\[ |\cC_1(G_0)| = \frac{\log(\theta^3 n)-5\log\log (\theta^3 n)+O_{\textsc{p}}(1)}{-\theta-\log(1-\theta)}\,.\]
\end{theorem}

\begin{theorem}[{\cite[Thm.~11]{Luczak98}}]\label{thm:Luc-diam}
For $G_0$ as above,
\[ \max_{i} \diam(\cC_i(G_0)) \leq \frac{\log(\theta^3 n) + O_{\textsc{p}}(1)}{-\log(1-\theta)}\,.\] 
\end{theorem}
We further need estimates on the average distance between all pairs of vertices that share a connected component, established by the next result.
\begin{lemma}\label{lem:subcritical_metric}
For $G_0$ as above, the random variable 
\[ Z:=\sum_{i} Z_i\qquad\mbox{where}\qquad Z_i := \sum_{u,v\in\cC_i(G_0)}\dist_{G_0}(u,v)\]
satisfies
\begin{align}
    \label{eq:E[Z]} 
    \E[Z]  &=  \left(1+O\left(\tfrac{1}{\theta^3 n}\right)\right)\frac{1-\theta}{\theta^2} n\,,
    \\
    \var(Z) &=O\left(\theta^{-7}n\right)\,.
    \label{eq:V(Z)}
\end{align}
\end{lemma}
\begin{proof}
We may rewrite $Z =\sum_u \sum^*_v \dist_{G_0}(u,v)$, where the sum $\sum^*$ goes over every vertex $v$ in the same connected component of $u$. With this in mind, let the random variable $X_k=X_k(u)$ be the number of vertices  at distance~$k$ in~$G_0$ from a given vertex $u$. Then by the symmetry of the vertices in $G_0$ playing the role of $u$,
\[
\E[Z]= n\sum_{k=0}^{\infty} k\,\E[X_k].
\]
The upper bound on $\E[Z]$ will be immediate: we claim $\E[X_k] \leq (1-\theta)^k$, and so
\begin{equation}
    \label{eqn:upperBoundEZ}
    \E[Z]\leq n\sum_{k} k(1-\theta)^k = \frac{(1-\theta)n}{\theta^2}\,. 
\end{equation}
Indeed, for each $v$, we have that $\P(\dist_{G_0}(u,v)=k) \leq n^{k-1}p_0^k$ as there are less than $n^{k-1}$ paths of length $k$ from $u$ to $v$; summing over $v$ gives $\E[X_k]\leq (1-\theta)^k$.

For the lower bound on $\E[Z]$, we will show that
\begin{equation}\label{eq:E[Xk]-lower}
\E[X_k] \geq (1-\theta)^k - 3k^2 \frac{(1-\theta)^{k+1}}{\theta n}\,,
\end{equation}
which, together with~\eqref{eqn:upperBoundEZ}, will yield~\eqref{eq:E[Z]} since $\sum k^3(1-\theta)^k=(\theta^2+6(1-\theta))\frac{1-\theta}{\theta^{4}}$.

To this end, we will need to first estimate $\E[X_l X_k]$ for $l\leq k$; if $(v_1,v_2)$ are such that $\dist_{G_0}(u,v_1)=l$ and $\dist_{G_0}(u,v_2)=k$, then there must be a pair of paths $P_1$ from $u$ to $v_1$ and $P_2$ from $u$ to $v_2$, of lengths $l$ and $k$ respectively, which intersect only in their first $j$ edges for some $0\leq j \leq l$. (This is due to the fact that any sub-path of a shortest path is itself a shortest path, thus from $P_1,P_2$ without this property we may generate $P'_1,P'_2$ as above by replacing the $j$-prefix of $P_1$ with that of $P_2$, where $j$ is the maximal index such that the $j$-th edge of $P_1$ is also that of $P_2$). For each value of $0\le j \le l$ there are at most $n^{k+l-j}$ such pairs of paths $P_1,P_2$, and the probability that two given such paths both appear in $G_0$ is at most $p_0^{k+l-j}.$ Hence, 
\begin{equation}\label{eq:E[XlXk]} \E[X_lX_k]\leq \sum_{j=0}^l (1-\theta)^{k+j} \leq \frac{(1-\theta)^k}\theta\,.\end{equation}
Next, let $(\cF_k)$ be the filtration corresponding to revealing the first $k$ levels of the breadth-first-search tree of $u$ in $G_0$, wherein the variable $X_k$ corresponds to the number of vertices in level $k$ of the tree. Conditional on $\cF_{k-1}$, the random variable $X_k$ is distributed as $\bin(n-\sum_{l=0}^{k-1} X_l, \eta)$ for $\eta = 1-(1-p_0)^{X_{k-1}}$. We use 
\[ 1-(1-p_0)^{X_{k-1}} \geq X_{k-1}\,p_0-\tbinom{X_{k-1}}2 p_0^2 \] to find that
\begin{align*}
\E[X_k] \ge~& np_0\E[X_{k-1}] - p_0\sum_{l=0}^{k-1}\E\left[X_l X_{k-1}\right]- \frac12 n p_0^2\, \E[X_{k-1}^2]\\
\ge~ & (1-\theta)\E[X_{k-1}] - \frac{(1-\theta)^{k+1}}{\theta}(k/n+p_0),
\end{align*}
which, by induction on $k$ starting from $X_0=1$, satisfies~\eqref{eq:E[Xk]-lower}, as claimed.

We turn to the bound on $\var(Z)$. Let us denote by $\cC_A$, for every subset $A\subseteq [n]$, the event that $A$ is the vertex set of a connected component in $G_0$, and let \[ Z_A:=\one_{\cC_A} \sum_{u,v\in A}\dist(u,v)\,.\]
Clearly, 
$
\sum_{i\ne j}Z_iZ_j = \sum_{A\ne B} Z_AZ_B,
$  
and $Z=\sum_AZ_A.$ In addition, $Z_A Z_B=0$ for every $A,B$ such that $A\cap B\ne\emptyset$; hence,
\[
\sum_{A\ne B} Z_AZ_B = \sum_{A} Z_A\sum_{B\subseteq[n]\setminus A}Z_B.
\]
The the random variables $Z_A$ and $\sum_{B\subseteq[n]\setminus A}Z_B$ are conditionally independent on the event $\cC_A$, 
and the latter is the analog of $Z$ for a random graph $F \sim \cG(n-|A|,p_0)$ whence its expectation is bounded from above by the right-hand side of (\ref{eqn:upperBoundEZ}). Therefore,
\begin{align*}
\E\Big[\sum_{A\ne B} Z_AZ_B\Big] &= \sum_A \E[Z_A\mid \cC_A]\P(\cC_A)\E\Big[\sum_{B\subseteq[n]\setminus A}Z_B \given \cC_A\Big] \\ &\le  
\E[Z] \frac{(1-\theta)n}{\theta^2} \,.
\end{align*}
(Although $\E[\sum_{B} Z_B \one_{\cC_A}]\leq \E[Z]$, we could not have increased $\E[\sum_{B\subseteq A^c} Z_B\mid \cC_A]$ to $\E[Z]$ due to the conditioning, and instead applied the upper bound in~\eqref{eqn:upperBoundEZ} on~$F$.)
We use this inequality and~\eqref{eq:E[Z]} (twice) to find that 
\[ \E\Big[\sum_{i\neq j} Z_i Z_j \Big] - \E[Z]^2 \leq \E[Z] \frac{(1-\theta)n}{\theta^2} O\Big(\frac{1}{\theta^3 n}\Big) \leq O(\theta^{-7} n)\,.
\]
The proof of~\eqref{eq:V(Z)} will thus be concluded once we show that
 \begin{equation}
 \E\Big[\sum_{i}Z_i^2\Big] = O\left(\theta^{-7}n\right)\,.
 \label{eqn:Zi2}
 \end{equation}
By Cauchy--Schwartz,
 \[
 Z_i^2 \leq  |\cC_i(G_0)|^2\!\!\sum_{u,v\in\cC_i(G_0)}\dist(u,v)^2\,.
 \]
 Hence, again by symmetry of the vertices in $G_0$ playing the role of $u$, 
 \begin{equation}\label{eq-sum-Zi^2-bound}
 \E\Big[\sum_i Z_i^2 \Big] = n\E\bigg[ \Big(\sum_{k\geq 0} X_k\Big)^2
 \Big(\sum_{k\geq 0} k^2 X_k\Big)
 \bigg]=
 n\sum_{k,l,m \geq 0} k^2\E\left[X_k X_l X_m
 \right]\,.
 \end{equation}
We will treat $\E[X_r X_s X_t]$ for $r\leq s \leq t$ similarly to~\eqref{eq:E[XlXk]}, by enumerating over the triples of vertices $v_1,v_2,v_3$ at respective distances $r,s,t$ from a fixed vertex $u$. Consider shortest paths $P_i$ from $u$ to $v_i$ for $i=1,2,3,$. If $0 \leq j \leq s$ is the number of common edges between $P_2$ and $P_3$, and $0 \leq i \leq r$ is the number of common edges between $P_1$ and $P_2$, then given $P_3$ one can construct $P_2$ followed by $P_1$ such that the $j$-prefix of $P_2$ coincides with that of $P_3$, and then the $i$-prefix of $P_1$ coincides with that of $P_2$.
Thus, there are at most $n^{t+(s-j)+(r-i)}$ choices for such $P_1,P_2,P_3$, and by the same token, at most $n^{t+(s-j)+(r-i)}$ choices for $P_1,P_2,P_3$ where the $i$-prefix of $P_1$ coincides with that of $P_3$. Altogether,
\[ \E[X_r X_s X_t] \leq 2\sum_{i=0}^r \sum_{j=0}^s (np_0)^{t+(s-j)+(r-i)} \leq \frac{2(1-\theta)^t}{\theta^2}\,.
\]
Splitting $\sum_k k^2 \sum_{l,m} \E[X_k X_l X_m]$ in of~\eqref{eq-sum-Zi^2-bound} into  $l,m \leq k$ (where we use the last display) and $k < l \vee m$ (where we enumerate on $l\vee m$ and apply the same bound),
\[
 \E\Big[\sum_i Z_i^2 \Big] \le 
 2n\sum_{k=0}^{\infty}k^2\bigg( k^2\frac{(1-\theta)^k}{\theta^2} +2\sum_{l=k+1}^{\infty} \frac{l(1-\theta)^l}{\theta^2}\bigg)\,,
\]
which is $O(n \sum k^4 \frac{(1-\theta)^k}{\theta^2} ) = O(\frac{n}{\theta^{7}})$, giving~\eqref{eqn:Zi2} and thus concluding the proof.
\end{proof}

\section{Proof of Theorem~\ref{thm:1}}\label{sec:conv-X}
\subsection{Noise sensitivity: proof of Theorem~\ref{thm:1}, Part~(\ref{it:sens})}
Let $G_0, G_1, G_1',G_1^\epsilon$ be the analogs of $\omega_0,\omega_1,\omega_1',\omega^\epsilon$ as defined above for the $N=\binom{n}2$ edge variables that correspond to $\cG(n,p)$, using $p_0$ and $p_1$ as in~\eqref{eq:p0}--\eqref{eq:omega}.
Conditioned on~$G_0$, the connected components of $G_1$ are constructed by a multiplicative coalescent process where two components $\cC_i(G_0),\cC_j(G_0)$ will are in $G_1$ with probability \[ 1-(1-p_1)^{|\cC_i(G_0)||\cC_j(G_0)|}\,.\]
Aldous~\cite{Aldous97} considered the following related random graph. For a vector of positive weights ${\bf x}=(x_1,\ldots,x_N)$ and $q>0$, let $\fG\sim\scW({\bf x},q)$ be the random graph on the vertices $\{1,\ldots,N\}$ where each edge $ij$ is present with probability $1-\exp(-q x_i x_j)$ independently of other edges. The weight of a connected component $\cC$ of $\fG$ is defined as $\wt(\cC) = \sum_{i\in \cC} x_i$, and the components of $\fG$ are ordered $(\cC_1(\fG),\cC_2(\fG),\ldots)$ in decreasing order of weights. With these notations, Aldous showed the following.
\begin{proposition}[{\cite[Prop.~4]{Aldous97}}]\label{prop:aldous}
Fix $\lambda\in\R$. For each $N$, let ${\bf x}(N)=(x_1,\ldots,x_N)$ be a positive vector and $q(N)>0$ be so that, if $\sigma_r := \| {\bf x}\|_r^r$ and $x_{\max} := \|{\bf x}\|_\infty$ then
\begin{equation}
\frac{\sigma_3}{(\sigma_2)^3}\to 1\,,\qquad q - \frac{1}{\sigma_2}\to\lambda\,,\qquad \frac{x_{\max}}{\sigma_2}\to 0\qquad\mbox{as $N\to\infty$}\,.
\label{eqn:aldous_assumptions}
\end{equation}
Then the random graph $\fG\sim\scW({\bf x},q)$ satisfies
\[ (\wt(\cC_j(\fG))\,:\;j\geq 1)\xrightarrow[]{\,\mathrm{d}\,} \bxi(\lambda)\quad\mbox{as $N\to\infty$}\,.\]
\end{proposition}
Observe that if let $\fG\sim\scW({\bf x},q)$ for ${\bf x}=(x_i)_{i\geq 1}$ given by \begin{equation}\label{eq:xi-q-def}
x_i=n^{-2/3}|\cC_i(G_0)|\qquad\mbox{and}\qquad q=n^{4/3}(-\log{(1-p_1)})\end{equation} 
(i.e., the vertices of $\fG$ are in correspondence with the connected component of $G_0$), then by the definitions of $\scW({\bf x},q)$ and construction of $G_1$, conditioned on $G_0$,
\begin{equation}\label{eq:W(x,q)-G1-equiv}
\left(\wt(\cC_j(\fG))\,:\;j\geq 1\right) \stackrel{\mathrm{d}}= n^{-2/3}\left(|\cC_j(G_1)|\,:\;j\geq 1\right) \,.
\end{equation}
With this in mind, we move to verify that the three conditions in~\eqref{eqn:aldous_assumptions} that qualify an application of Proposition~\ref{prop:aldous} hold for $G_0$ w.h.p.

\begin{lemma}\label{lem:G0-aldous-cond}
Let $G_0\sim\cG(n,p_0)$ with $p_0$ as in~\eqref{eqn:p_0_spelled_out} for $\lambda\in \R$ fixed and $0<\epsilon(n)<1$ such that $\epsilon^3 n$ goes to $\infty$ with $n$. Let ${\bf x},q$ be as in~\eqref{eq:xi-q-def}. Then w.h.p.,
\begin{align}
\label{eq:G0-aldous-cond-1}	
&\Big|\frac{\sigma_3}{(\sigma_2)^3}- 1\Big|<  (\epsilon^3 n)^{-1/5}\,,\\
\label{eq:G0-aldous-cond-2}
&\Big|q-\frac{1}{\sigma_2}- \lambda\Big| <  (\epsilon^3 n)^{-1/15}\,,\\
\label{eq:G0-aldous-cond-3}
&\frac{x_{\max}}{\sigma_2} < 4(\epsilon^3 n)^{-1/3}\log (\epsilon^3 n)\,.
\end{align}
\end{lemma}
\begin{proof}
Let $\theta=1-n p_0$ be as in~\eqref{eq:theta}, and write $\xi := \theta^3 n$ for brevity (recalling that $\theta = (1-o(1))\epsilon$ whence $\theta^3 n\to\infty$). We apply Theorem~\ref{thm:Jan_Luc_susceptibility} to find that
the variables $\sigma_2 = n^{-4/3}\sum_i |\cC_i(G_0)|^2$ and $\sigma_3 = n^{-2}\sum_i |\cC_i(G_0)|^3$ satisfy
 \begin{align*} \E[\sigma_2 ] 
 = \xi^{-1/3} + O(\xi^{-4/3})\,,
 \qquad &
  \E[\sigma_3] 
  = \xi^{-1} + O(\xi^{-2})
  \end{align*}
by~\eqref{eq:E-C^r}, whereas by~\eqref{eq:Var-C^r},
 \begin{align*} \var(\sigma_2) 
 = O(\xi^{-5/3})\,,
 \qquad &
  \var(\sigma_3) 
  = O(\xi^{-3})\,.
 \end{align*}
By Chebyshev's inequality, we see that the events
\begin{align*} \cE_2 = \{ |\sigma_2 - \xi^{-1/3}| \leq \xi^{-3/4} \}\,,\qquad&
\cE_3 = \{ |\sigma_3 - \xi^{-1}| \leq \xi^{-5/4}\}
\end{align*}
occur w.h.p., and specifically
\begin{align}\label{eq:P(E2)-P(E3)} 
\P( \cE_2^c ) \leq O(\xi^{-1/6}) = o(1)\,,\qquad &
\P( \cE_3^c ) \leq O(\xi^{-1/2}) = o(1)\,.
 \end{align}
For $G_0\in \cE_2 \cap \cE_3$ we immediately see that
 \[
  \frac{\sigma_3}{(\sigma_2)^3} = \frac{(1+O(\xi^{-1/4}))\xi^{-1}}{(1+O(\xi^{-5/12}))\xi^{-1}} = 1+O((\epsilon^3 n)^{-1/4})\,,
 \]
 and in particular~\eqref{eq:G0-aldous-cond-1} holds for large enough $n$. 
 For the second (somewhat subtler) inequality~\eqref{eq:G0-aldous-cond-2},  recall $\theta = \epsilon - (1-\epsilon)\lambda n^{-1/3} + O(1/n)$ and  $p_1=\epsilon p=\epsilon(1+\lambda n^{-1/3})/n$. Thus, for every $G_0\in\cE_2$,
\begin{align*}
q - \sigma_2^{-1} &= (1 + O(\epsilon/n)) n^{1/3} \epsilon (1+\lambda n^{-1/3}) - (1+O(\xi^{-5/12}))\theta n^{1/3} \\
&= \lambda + O\big(\xi^{-5/12}\theta n^{1/3}\big) + O( n^{-2/3}) = \lambda + O( (\epsilon^3 n)^{-1/12} )\,,
\end{align*}
yielding~\eqref{eq:G0-aldous-cond-2}.
Finally, from Theorem~\ref{thm:Bol-Luc-C1} we have  $n^{2/3} x_{\max} \leq  2( \log\xi + O_{\textsc{p}}(1))/\theta^2$, and so $G_0$ satisfies both $\cE_2$ and  $x_{\max} \leq 3 \xi^{-2/3}\log\xi$ w.h.p., in which case
\[ 
\frac{x_{\max}}{\sigma_2} \leq (3+o(1)) \xi^{-1/3}\log\xi < 4(\epsilon^3 n)^{-1/3}\log(\epsilon^3 n)
\]
for large  $n$, as claimed in~\eqref{eq:G0-aldous-cond-3}. This completes the proof.
\end{proof}
Combining Proposition~\ref{prop:aldous} with Lemma~\ref{lem:G0-aldous-cond} we find that for each $n\geq 1$ there exists a set of graphs $\cA_0^{(n)}$ such that $\P(G_0\in\cA_0^{(n)})\to 1$, and for every $G\in\cA_0^{(n)}$, the conditional law of $\bX(G_1)$ given that $G_0=G$ converges in distribution to $\bxi(\lambda)$. Namely, for every continuity set $S$ of $\P(\bxi(\lambda)\in\cdot)$,
\[
     \lim_{n\to\infty} \max_{G\in \cA_0^{(n)}} \left| \P_n\left( \bX_n(G_1) \in S \mid G_0=G\right) - \P(\bxi(\lambda)\in S)\right| = 0\,.
\]
In particular (cf.~Remark~\ref{rem:common-core-weaker}), the hypothesis~\eqref{eq:common-core-full-hyp} in Corollary~\ref{cor:common-core-full} holds true, and hence we may conclude from it that the pair $({\bf X}(G_1),{\bf X}(G_1^\epsilon))$ converges in distribution to a product of two i.i.d.\ copies of $\bxi(\lambda)$.
\qed

\subsection{Noise stability: proof of Theorem~\ref{thm:1}, Part~(\ref{it:stab})}
Recall the setup of $G_0,G_1,G_1',G_1^\epsilon$ from the previous sections. The hypothesis $\epsilon^3 n\to 0$ means that~$G_0$ is now also a critical random graph satisfying
$\nu_n^0=n^{-2/3}(|\cC_j(G_0)\,:\;j\ge 1) \xrightarrow{\mathrm d} \bxi(\lambda)$.
We need the following two results of Aldous~\cite{Aldous97} for the random graph model $\scW({\bf x},q)$.
Recall $\cC_j(\fG)$ is the component of a weighted graph $\fG$ with the $j$-th largest weight. 
\begin{lemma}[{\cite[Lem.~17]{Aldous97}}]\label{lem:Aldous-G-Gbar}
If $\fG,\bar\fG$ are weighted graphs on $N$ vertices with resp.\ weights ${\bf x}=(x_i)_{i=1}^N,\bar{\bf x}=(\bar{x}_i)_{i=1}^N$ such that $E(\fG)\subseteq E(\bar\fG)$ and $x_i\leq \bar{x}_i$ for all $i$, then
\[
\sum_{j}\left(\wt(\cC_j(\bar\fG))-\wt(\cC_j(\fG))\right)^2 \le
\sum_{j}\wt(\cC_j(\bar\fG))^2 -
\sum_{j}\wt(\cC_j(\fG))^2\,.
\]
\end{lemma}
\begin{lemma}[{\cite[Lem.~20]{Aldous97}}]\label{lem:Aldous-sigma2}
The weighted random graph $\fG\sim\scW({\bf x},q)$ satisfies
\[
\P\Big(\sum_{j} \wt(\cC_j(\fG))^2 > s\Big) \leq 
\frac{q s \sigma_2}{s-\sigma_2}
\]
for every $s>\sigma_2=\sum_jx_j^2$.
\end{lemma}
\begin{corollary}
\label{cor:stab_sizes}
If $\epsilon^3 n\to 0$ and $G_0,G_1$ are as above, then 
$\|\bX(G_1)-\bX(G_0)\|_2 \xrightarrow{\,\mathrm{p}\,} 0$.
\end{corollary}
\begin{proof}
Let $\delta>0$. We let $\bar\fG\sim\scW({\bf x},q)$ for $\bf x$ and $q$ as in~\eqref{eq:xi-q-def}, noting that 
\[q = (1+o(1))\epsilon n^{1/3}\,,\]
while recalling that $\bX(G_1)\stackrel{\mathrm{d}}=n^{-2/3}(\wt(\cC_j(\bar\fG)))_{j\geq 1}$ conditioned on $G_0$ as per~\eqref{eq:W(x,q)-G1-equiv}.
Further let $\fG$ be the edgeless graph $\fG$ on a vertex set corresponding to the connected components of $G_0$ and the same weights ${\bf x}$ (whereby $\sum x_i^2 = \sum_j\wt(\cC_j(\fG))^2)$. 
 Applying Lemma~\ref{lem:Aldous-G-Gbar} for $\fG,\bar\fG$ we deduce that, conditioned on $G_0$,
\[ 
\P\big(\|\bX(G_1)-\bX(G_0)\|_2^2 > \delta\mid G_0\big) \leq \P\Big(\sum_{j}\wt(\cC_j(\bar\fG))^2>\sigma_2+\delta\Big)\,,
\]
at which point Lemma~\ref{lem:Aldous-sigma2} further shows that
\[
\P\big(\|\bX(G_1)-\bX(G_0)\|_2^2 > \delta\mid G_0\big) \leq \frac{q(\sigma_2+\delta)\sigma_2}{\delta}\,.
\]
Recall that $n^{-2/3}(|\cC_j(G_0)|\,:\;j\geq 1)$ converges weakly at $n\to\infty$ to a nontrivial limit in the $\ell^2$-topology, whence $\sigma_2 = n^{-4/3}\sum |\cC_j(G_0)|^2$ converges in distribution precisely to the $\ell^2$-norm, squared, of the random limit $\bxi(\lambda)$ in the Polish space~$\selltwo$. As such, $\|\bxi(\lambda)\|_2$ is tight, hence so is $\sigma_2^2$, while $q\to 0$ by assumption, and this completes the proof.
\end{proof}
Part~\eqref{it:stab} of Theorem~\ref{thm:1}  follows directly from Corollary~\ref{cor:stab_sizes}, applied to $(G_0,G_1)$ and again to $(G_0,G_1')$ (which is equal in law), and using that $(G_0,G_1^\epsilon)\stackrel{\mathrm{d}}=(G_0,G_1')$.
\qed

\section{Proof of Theorem~\ref{thm:2}}\label{sec:conv-M}

\subsection{Noise sensitivity: proof of Theorem~\ref{thm:2}, Part~(\ref{it:sens})}
Our approach here is similar to the proof of Theorem~\ref{thm:1}, building on a special case of the metric-space extension of Proposition~\ref{prop:aldous} that was proved in ~\cite{BBSW14}. In addition to ${\bf x}=(x_1,\ldots,x_N)$ and $q>0$ as before, let $\bH=(H_1,\ldots,H_N)$ be finite connected graphs. Define  $\fG\sim\scW(\mathbf{x},q)$ as before, whereby the weight of a connected component $\cC$ in it is $\wt(\cC)=\sum_{i\in\cC}x_i$. Let $\fG^\bH\sim\scW(\mathbf{x},\bH,q)$
be the (unweighted) random graph obtained by generating a weighted random graph $\fG\sim\scW(\mathbf{x},q)$, replacing each vertex $i$ of $\fG$ with the graph~$H_i$, and 
connecting a uniform random vertex of $H_i$ with a uniform random vertex of~$H_j$ for every edge~$ij\in E(\fG)$.
We let $\cC_j(\fG^\bH)$ be the component of $\fG^\bH$ corresponding to $\cC_j(\fG)$, the component of $\fG$ with the $j$-th largest weight, and view it as a measured metric space $\tilde{\sM}_j(\fG^\bH)$ via the scaled counting measure
\[
\mu_j(A) = \sum_{i\in \cC_j(\fG)} x_i\cdot\frac{|A\cap V(H_i)|}{|V(H_i)|}\qquad \mbox{for all $A\subseteq \cC_j(\fG^\bH)$}\,,
\]
(N.B.\ $\wt(\cC_j(\fG))$ coincides with its total measure), and the graph metric scaled by
\[ \mathfrak{s}(\mathbf{x},\bH) := \frac{(\sigma_2)^2}{\sigma_2+\sum_{i=1}^{N}x_i^2u_i}\,,\] 
where $\sigma_2=\|{\bf x}\|_2^2$ and  $u_i$ is the expected distance between two independent uniform random vertices in $H_i$. (We distinguish the notation $\tilde\sM_j(\fG^\bH)$ from $\sM_j(G)$ since the scaling factor here is a function of $\mathbf{x},\bH$ as opposed to scaling by $\mathfrak{s}=|V(G)|^{-1/3}$).

With this notation, following is the metric space analog due to~\cite{BBSW14} of Prop.~\ref{prop:aldous}.
\begin{proposition}[{\cite[Thm.~3.4]{BBSW14}}]
Fix $\lambda\in\R$. For each $N$, let ${\bf x}(N)=(x_1,\ldots,x_N)$ be a positive vector and $q(N)>0$ so that, if $\sigma_r:=\|{\bf x}\|_r^r$ and $x_{\max}:=\|{\bf x}\|_{\infty}$ then assumption (\ref{eqn:aldous_assumptions}) holds true. 
In addition, 
let ${\bf H}(N)=(H_1,\ldots,H_N)$ be a sequence of disjoint finite connected graphs so that, if $x_{\min}:=\min_{i}x_i$, $d_{\max}:=\max_i \diam(H_i)$, and $u_i$ is the expected distance of two independent uniform random vertices in $H_i$ then there exist some positive real numbers $r_0>0$ and $0<\eta_0<\frac12$ such that
\begin{equation}
\frac{x_{\max}}{(\sigma_2)^{3/2+\eta_0}}\to 0\,,\quad 
\frac{(\sigma_2)^{r_0}}{x_{\min}}\to 0\,,\quad 
\frac{(\sigma_2)^{3/2-\eta_0}d_{\max}}{\sum_{i=1}^{N}x_i^2u_i+\sigma_2}\to 0\,,\quad 
\frac{\sigma_2x_{\max}d_{\max}}{\sum_{i=1}^{N}x_i^2u_i}\to 0
\label{eqn:metric_assumptions}
\end{equation}
as $N\to\infty$. Then the random graph $\fG^\bH\sim\scW(\mathbf{x},\bH,q)$ satisfies, for every $j\ge 1$,
\[
\tilde\sM_j(\fG^\bH)\xrightarrow{\,\mathrm d\,} \fM_j(\lambda)\quad\mbox{as $N\to\infty$ in $(\scM,d_{\GHP})$}\,.\]
\label{prop:metric_bbsw}\end{proposition}

The proof of Theorem~\ref{thm:2},~\eqref{it:sens2} consists of several ingredients. We start by setting
\begin{equation}\label{eq:x-q-H}
x_i=n^{-2/3}|\cC_i(G_0)|\,, \quad q=n^{4/3}(-\log{(1-p_1)})\,,\quad H_i=\cC_i(G_0)\,.\end{equation}
Noting that the $x_i$'s and $q$ are exactly as they were defined in~\eqref{eq:xi-q-def}, we immediately infer that the conditions in~\eqref{eqn:aldous_assumptions} hold for $G_0\sim\cG(n,p_0)$ w.h.p.\ using Lemma~\ref{lem:G0-aldous-cond}. The following result confirms that the additional conditions in~\eqref{eqn:metric_assumptions}, required for an application of Proposition~\ref{prop:metric_bbsw}, also hold for the graph $G_0$ w.h.p. 
\begin{lemma}\label{lem:G0-metric-prop}
Fix $\lambda\in\R$ and $\delta>0$. 
Let $G_0\sim\cG(n,p_0)$ with $p_0$ as in~\eqref{eqn:p_0_spelled_out} for $\epsilon(n)$ which satisfies $ n^{-1/3+\delta}\leq\epsilon \leq 1-n^{-1/2}$. Further define $(x_i)$ and $(H_i)$ as in~\eqref{eq:x-q-H}. Then for every fixed $0<\eta_0<\frac12$ and $r_0>1/\delta$, the following inequalities hold w.h.p.
\begin{align}
    \label{eq:G0-bbsw-1}
&\frac{x_{\max}}{(\sigma_2)^{3/2+\eta_0}} \leq 5 (\epsilon^3n)^{-(1-2\eta_0)/6}\log(\epsilon^3n)\,,\\
    \label{eq:G0-bbsw-2}
&\frac{(\sigma_2)^{r_0}}{x_{\min}}\le 2n^{-1/3}\,,\\
    \label{eq:G0-bbsw-3}
&\frac{(\sigma_2)^{3/2-\eta_0}d_{\max}}{\sigma_2+\sum_{i=1}^{m}x_i^2u_i}\le 
5(\epsilon^3n)^{-(1-2\eta_0)/6}\log(\epsilon^3n)\,,\\
    \label{eq:G0-bbsw-4}
&\frac{\sigma_2x_{\max}d_{\max}}{\sum_{i=1}^{m}x_i^2u_i}\le
10(\epsilon^3n)^{-2/3}\log^2(\epsilon^3 n)\;\wedge\;n^{-1/8}\,,\\
    \label{eq:G0-bbsw-5}
&\left|\mathfrak{s}(\mathbf{x},\bH) n^{1/3}-1\right| \leq 4(\epsilon^3n)^{-2/5}\,.
\end{align}
\end{lemma}

\begin{proof}
Observe that the random variable $\sum_{i=1}^{N}x_i^2u_i$ is equal to $n^{-4/3}Z$ for $Z$ as defined in Lemma~\ref{lem:subcritical_metric}. We recall that $\theta = 1-np_0=\epsilon+O(n^{-1/3})=(1+o(1))\epsilon$ as per~\eqref{eq:theta}, and can therefore infer from that lemma that
\begin{align*}
\E\Big[\sum_ix_i^2u_i\Big] &=\left(1+O\big((\epsilon^3 n)^{-1/3}\big)\right) \frac{1-\theta}{\theta^2 n^{1/3}}\,,\\
\var\Big(\sum_ix_i^2u_i\Big) &= O\left(\epsilon^{-7}n^{-5/3}\right)\,.
\end{align*}
By Chebyshev's inequality, 
\begin{equation}
\label{eqn:sumxi2ui}
\P\bigg( \Big| \sum_ix_i^2u_i - \frac{1-\theta}{\theta^2 n^{1/3}}\Big| > \frac{1}{(\epsilon^3 n)^{2/5}\epsilon^2 n^{1/3}}\bigg) = O\left( (\epsilon^3 n)^{-1/5}\right) = o(1)\,.
\end{equation}
We derive the proof from this bound, along with those of Lemma~\ref{lem:G0-aldous-cond} and the fact  \[ d_{\max}\leq 2\epsilon^{-1} \log(\epsilon^3 n)\,,\]
which holds w.h.p.\ by Theorem~\ref{thm:Luc-diam},
as  $d_{\max} \leq  ( \log(\epsilon^3 n) + O_{\textsc{p}}(1))/\theta$.

For~\eqref{eq:G0-bbsw-1}, recall from~\eqref{eq:P(E2)-P(E3)} that, w.h.p., the event $\cE_2$ holds, on which one has that $\sigma_2  = (1+o(1)) (\theta^3 n)^{-1/3} = (1+o(1))(\epsilon^3 n)^{-1/3}$.
Plugging 
\[ (\sigma_2)^{\frac12+\eta_0} = (1+o(1))(\epsilon^3 n)^{-\frac16-\frac{\eta_0}3}\] in~\eqref{eq:G0-aldous-cond-3}, we immediately see that~\eqref{eq:G0-bbsw-1} holds w.h.p.

For~\eqref{eq:G0-bbsw-2}, on the aforementioned event $\cE_2$ we have $(\sigma_2)^{r_0 } \le 2(\epsilon^3 n)^{-r_0/3} \leq 2n^{-\delta r_0}$ for large enough $n$, by the assumption $\epsilon\geq n^{-1/3+\delta}$ (this is the only place in the proof where we use this stronger assumption rather than requiring that $\epsilon^3 n\to \infty$); since $r_0\geq 1/\delta$,   
whereas $x_{\min}\ge n^{-2/3}$ holds deterministically, we arrive at~\eqref{eq:G0-bbsw-2}.

For~\eqref{eq:G0-bbsw-3}, the above bounds on $\sigma_2$ and $d_{\max}$ show that, w.h.p.,
\[ (\sigma_2)^{3/2-\eta_0} d_{\max} \leq (2+o(1))\epsilon^{-1}(\epsilon^3 n)^{-1/2+\eta_0/3}\log(\epsilon^3 n)\,,
\]
while $\sum_ix_i^2 u_i = (1+o(1))(1-\epsilon)\epsilon^{-2}n^{-1/3}$ and $\sigma_2 = (1+o(1))\epsilon^{-1}n^{-1/3}$. If $\epsilon<\frac12$ then $\sum_i x_i^2 u_i\geq (\frac12 -o(1))\epsilon^{-2}n^{-1/3}$, while if $\epsilon\geq \frac12$ then $\sigma_2 \geq (\frac12-o(1)) \epsilon^{-2}n^{-1/3}$, hence in both cases
we can combine this bound with the last display to obtain~\eqref{eq:G0-bbsw-3}.

The bound in~\eqref{eq:G0-bbsw-4} follows from plugging in the above bounds on  $\sigma_2$ and $d_{\max}$ in~\eqref{eq:G0-aldous-cond-3}, along with the asymptotics $\sum_i x_i^2 u_i$ established above, to see that w.h.p.
\[\frac{\sigma_2 x_{\max} d_{\max}}{\sum_i x_i^2 u_i}  \leq (8+o(1))\frac{\epsilon^{-1} (\epsilon^3 n)^{-1}\log^2(\epsilon^3 n)}{(1-\epsilon)\epsilon^{-2} n^{-1/3}}=\frac{8+o(1)}{1-\epsilon}(\epsilon^3 n)^{-2/3}\log^2(\epsilon^3 n)\,.
\]
This gives the first bound in the right-hand of~\eqref{eq:G0-bbsw-4} if we have $\epsilon<\frac1{10}$, while in the case $\frac1{10}\leq \epsilon \leq 1-n^{-1/2}$ we arrive at the bound $n^{-1/6+o(1)}<n^{-1/8}$ for large $n$.

To obtain the final inequality~\eqref{eq:G0-bbsw-5}, we combine the bound~\eqref{eqn:sumxi2ui} with the fact that $|\sigma_2 - (\theta n^{1/3})^{-1}|\leq (\theta^3 n)^{-3/4}$ on the event~$\cE_2$, and deduce that w.h.p.,
\[
\Big| \sigma_2 + \sum_i x_i^2 u_i - \frac1{\theta^{2}n^{1/3}}\Big| \leq \frac1{(\theta^3 n)^{3/4}}+\frac{(\epsilon^3 n)^{-2/5}}{\epsilon^2 n^{1/3}} \leq
 \frac{2(\epsilon^3 n)^{-2/5}}{\epsilon^{2}n^{1/3}}\,,
\]
where we absorbed the term $(1+o(1))\epsilon(\epsilon^3 n)^{-5/12}$ into $(\epsilon^3 n)^{-2/5}$ for large enough~$n$.
The above bound on $\sigma_2$ given the event $\cE_2$ further implies that
\[
\Big|\frac{(\sigma_2)^2}{ (\theta^3 n)^{-2/3}} - 1\Big| \leq (2+o(1))(\theta^3 n)^{-5/12} < (\epsilon^3 n)^{-2/5}
\] 
for a sufficiently large $n$. Combined,  we obtain that $\mathfrak{s}(\mathbf{x},\bH)$ is bounded w.h.p.\ via
\[
 \bigg|\frac{(\sigma_2)^2 n^{1/3}}{\sigma_2+\sum_i x_i^2u_i}-1\bigg| \leq \frac{3(\epsilon^3 n)^{-2/5}}{1-2(\epsilon^3 n)^{-2/5}} \leq 4(\epsilon^3 n)^{-2/5}\,,
\]
concluding the proof.
\end{proof}

The second step in the proof is to argue that w.h.p.\ over the random graph $G_0$, the graphs $G_1$ conditioned on $G_0$ and $\fG^\bH\sim\scW(\mathbf{x},\bH,q)$ can be coupled such that $d_{\GHP}(\sM_j(G_1),\tilde\sM_j(\fG^\bH)) \xrightarrow{\,\mathrm{p}\,} 0$. This is the (approximate) metric analogue of the distributional equality~\eqref{eq:W(x,q)-G1-equiv} that was used in the proof of Theorem~\ref{thm:1}. The reason that such a distributional equality does not hold here is twofold. First, $G_1$ can contain multiple edges between two components of $G_0$ as well as additional edges within a component of $G_0$. Second, the scaling factor $\mathfrak{s}(\mathbf{x},\bH)$ used for $\tilde\sM_j(\fG^\bH)$ differs from the $n^{-1/3}$ scaling of $\dist_{G_1}$ for $\sM_j(G_1)$. 
On the other hand, note that for the choice of $\mathbf{x},\bH$ as in~\eqref{eq:x-q-H}, the measure associated with both spaces is the same, namely, the counting measure multiplied by $n^{-2/3}$ (as the factor $|V(H_i)|$ in the definition of $\mu_j$ cancels out with the one in $x_i$).

\begin{lemma}
\label{lem:coupling_metric_spaces}
Fix $\lambda\in\R$. Suppose that $p = (1+\lambda n^{-1/3})/n$, let $0<\epsilon<1$ be such that $\epsilon^3 n \to \infty$ as $n\to\infty$, and define $G_0,G_1$ to be as above. Then one can construct a random graph $\fG^\bH$ such that  $G_0\subseteq\fG^\bH\subseteq G_1$ and so that the following hold.
\begin{enumerate}[(i)]
    \item $\fG^\bH\sim\scW(\mathbf{x},\bH,q)$ where $\mathbf{x}$, $\bH$ and $q$ are as in~\eqref{eq:x-q-H}.
    \item There exist a set of graphs $\cB_0^{(n)}$ such that $\P(G_0\in \cB_0^{(n)})\to 1$ as $n\to\infty$, and for every graph $G\in\cB_0^{(n)}$ and $j\ge 1$, 
    \[\E\Big[d_{\GHP}\big(\sM_j(G_1),\tilde\sM_j(\fG^\bH)\big) \given G_0=G\Big]\to 0 \qquad\mbox{as $n\to\infty$}\,.\]
\end{enumerate}
\end{lemma}
\begin{proof}
Given $G_0$ and $G_1$, we call a pair of connected components $\cC_i(G_0)$ and $\cC_j(G_0)$ {\em excessive in $G_1$} if they are connected in $G_1$ by more than one edge. Similarly, we call a single component $\cC_i(G_0)$ excessive if $G_1$ contains an edge within $\cC_i(G_0)$ (i.e., an edge between two of its vertices) that does not belong to $G_0$. To generate $\fG^\bH$ from~$G_1$, we remove from each excessive component $\cC_i(G_0)$ every edge that does not belong to $G_0$, and further remove all the edges between every pair of excessive components $\cC_i(G_0),\cC_j(G_0)$ except one that is chosen uniformly at random.

The graph $\fG^\bH$ is a $\scW(\mathbf{x},\bH,q)$-distributed subgraph of $G_1$ with the exact same connected components. If we condition on $G_0$, the probability of the event $\mathfrak{X}(i,j)$ that a pair $\cC_i(G_0)$ and $\cC_j(G_0)$ is excessive is at most $(|\cC_i(G_0)||\cC_j(G_0)|p_1)^2$, and the probability of the event $\mathfrak{X}(i)$ that $\cC_i(G_0)$ is excessive is at most $|\cC_i(G_0)|^2p_1$. Therefore, as $p_1=(1+o(1))\epsilon/n$, the variable $X:=\sum_{i}\one_{\mathfrak{X}(i)} + \sum_{i<j}\one_{\mathfrak{X}(i,j)}$ satisfies
\begin{equation}\label{eq:X|G0}
\E[X\mid G_0]\leq (1+o(1))\Big(\frac{\epsilon}n \sum_{i}|\cC_i(G_0)|^2\Big)^2+(1+o(1))\frac{\epsilon}n\sum_{i}|\cC_i(G_0)|^2\,.
\end{equation}

For every pair of vertices $u,v\in [n]$, consider a shortest path $P=(x_1,\ldots,x_m)$ between them in $G_1$, and suppose that $P$ uses an edge between a vertex $x_l\in \cC_i(G_0)$ and $x_{l+1}\in \cC_j(G_0)$ that does not belong to $\fG^\bH$. 
Let $r$ be the largest index such that $x_r\in \cC_j(G_0)$. If $i=j$ then we replace the segment $x_l,\ldots,x_r$ in $P$ with a shortest path in $G_0$ from $x_l$ to $x_r$. If on the other hand $i\ne j$, denote by $yz$ the edge in $\fG^\bH$ that connects $\cC_i(G_0)$ and $\cC_j(G_0)$, and replace the segment $x_l,\ldots,x_r$ by a shortest path from $x_l$ to $y$ in $G_0$, followed by the edge $yz$ and then by a shortest path from~$z$ to $x_r$ in $G_0$. Repeat this procedure until a path in $\fG^\bH$ is obtained. The number of detours does not exceed the number of excessive pairs and components, and each such detour increases the length of the path by at most $2d_{\max}(G_0)$. Therefore,
\[
0\leq \dist_{\fG^\bH}(u,v) - \dist_{G_1}(u,v)  \leq 2d_{\max}(G_0)X\,.
\]
Since this bound holds true for every two vertices $u,v$, and the measure on $\sM_j(G_1)$ identifies with the measure on $\tilde\sM_j(\fG^\bH)$, we find, by taking the scaling factors of the distances into consideration,  that 
\[
d_{\GHP}(\sM_j(G_1),\tilde\sM_j(\fG^\bH)) \le 
2\mathfrak{s}(\mathbf{x},\bH) d_{\max}(G_0) X + |\mathfrak{s}(\mathbf{x},\bH)-n^{-1/3}|\diam(\cC_j(\fG^\bH))\,.
\]
Define $\cB_0^{(n)}$ to be the set of every graph $G$ such that, for $\mathbf{x},\bH,q$ as in~\eqref{eq:x-q-H} with $G$ playing the role of $G_0$, and $\sigma_r$, $x_{\max}$, $x_{\min}$ and $d_{\max}$ defined as in Proposition~\ref{prop:metric_bbsw}, the following conditions are satisfied:
\begin{itemize}
    \item Inequalities~\eqref{eq:G0-bbsw-1}--\eqref{eq:G0-bbsw-5} of Lemma~\ref{lem:G0-metric-prop} hold;
    \item $\sigma_2 \leq 2 / (\epsilon n^{1/3})$; 
    \item $d_{\max}\leq 2\log(\epsilon^3 n)/\epsilon$.
\end{itemize}
By Lemma~\ref{lem:G0-metric-prop} and Theorems~\ref{thm:Jan_Luc_susceptibility}
and~\ref{thm:Luc-diam}, respectively, we have $\P(G_0\in\cB_0^{(n)})\to 1$. Moreover, for every $G\in\cB_0^{(n)}$, we infer from~\eqref{eq:G0-bbsw-5} and the condition on $d_{\max}$ that
\begin{align*}
\E\Big[d_{\GHP}(\sM_j(G_1),&\tilde\sM_j(\fG^\bH))\given G_0=G\Big] \leq (2+o(1)) (\epsilon^3 n)^{-1/3}\log(\epsilon^3 n) \E[X\mid G_0] \\&+ 4(\epsilon^3 n)^{-2/5}\E\left[n^{-1/3}\diam(\cC_j(\fG^\bH))\given G_0=G\right]
\,.
\end{align*}
For all $G\in\cB_0^{(n)}$, Proposition~\ref{prop:metric_bbsw} implies  $\E[n^{-1/3}\diam(\cC_j(\fG^\bH))\mid G_0=G]=O(1)$, and $\E[X\mid G_0=G]=O(1)$ using $\sum_i\cC_i(G)|^2\le 2n/\epsilon$ combined with~\eqref{eq:X|G0}.
It thus follows that $\E[d_{\GHP}(\sM_j(G_1),\tilde\sM_j(\fG^\bH))\mid G_0=G]\to 0$ as $n\to\infty$, as claimed.
\end{proof}

To complete the proof of Theorem \ref{thm:2},~\eqref{it:sens2}, first consider $n^{-1/3+\delta}<\epsilon<1-n^{-1/2}$. By combining Proposition~\ref{prop:metric_bbsw} with Lemmas \ref{lem:G0-metric-prop} and \ref{lem:coupling_metric_spaces} we find that for each $n\ge 1$ there exists a set of graph $\tilde\cB_0^{(n)}$ so that $\P(G_0\in\tilde\cB_0^{(n)})\to 1$, and for every $G\in\tilde\cB_0^{(n)}$,
\begin{enumerate}
    \item The measured metric space $\tilde\sM_j(\fG^\bH)$ conditioned on $G=G_0$ converges in distribution to $\fM_j(\lambda)$  w.r.t.\ \GHP\ distance. 
    \item The spaces $\tilde\sM_j(\fG^\bH)$ and $\sM_j(G_1)$, conditioned on $G_0=G$, can be coupled such that their \GHP\ distance vanishes in probability.
\end{enumerate}
Therefore, by Slutsky's Lemma for Polish spaces (see, e.g.,\cite[Thm.~3.1]{BIL99}), for every $\P(\fM_j(\lambda)\in\cdot)$-continuity Borel set $S\subseteq \scM$,
\[
     \lim_{n\to\infty} \max_{G\in \tilde\cB_0^{(n)}} \left| \P_n\left( \sM_j(G_1)\in S \mid G_0=G\right) - \P(\fM_j(\lambda)\in S)\right| = 0\,.
\]
By Corollary \ref{cor:common-core-full}, this completes the first assertion in Theorem~\ref{thm:2},~\eqref{it:sens} in this regime.

For $\epsilon \geq 1-n^{-1/2}$, this assertion will follow directly from Theorem~\ref{thm:2}, Part~\eqref{it:stab2}, whose proof is in~\S\ref{sec:stab2} below. Indeed, we establish there that if $\tilde\epsilon$ satisfies $\tilde\epsilon^3 n\to 0$, and $\tilde G_0$ and $\tilde G_1$ are random graphs distributed as 
\[ \tilde G_0\sim\cG\big(n,\tfrac{1-\tilde\epsilon}{1-\tilde\epsilon p}p\big)\quad,\quad 
\tilde G_2\sim\cG\left(n,\tilde\epsilon p\right)\,,\] then their union $\tilde G_1=\tilde G_0 \cup \tilde G_2$ satisfies $ d_{\GHP}(\sM_j(\tilde G_0),\sM_j(\tilde G_1))\xrightarrow{\,\mathrm p\,} 0$ as $n\to\infty$.
By Markov's inequality, this  holds conditionally on  $\tilde G_2$ such that $\tilde G_2\in\cK$, for some set of graphs $\cK$ with $\P(\tilde G_2\in\cK)\to 1$.
In addition, $\sM_j(\tilde G_0)\xrightarrow{\,\mathrm d\,}\fM_j(\lambda)$ as $n\to\infty$ since $(n\frac{1-\tilde\epsilon}{1-\tilde\epsilon p}p-1)n^{1/3}\to\lambda$. 
It follows that for every $\tilde G_2\in\cK$, the conditional law of $\sM_j(\tilde G_1)$ given $\tilde G_2$ weakly converges  to that of $\fM_j(\lambda)$.
We now flip the roles of the graphs by setting $\tilde \epsilon=\frac{1-\epsilon}{1-\epsilon p}$. Observe that $\tilde \epsilon^3 n \to 0$, $G_0=\tilde G_2$ and $G_1 = \tilde G_0 \cup \tilde G_2$. Thus, the conditional law of $\sM_j(G_1)$ given $G_0$ such that~$G_0\in\cK$ converges in distribution to $\fM_j(\lambda)$ as $n\to\infty$, and Corollary~\ref{cor:common-core-full} concludes the proof.

The second assertion of Theorem~\ref{thm:2}, Part~\eqref{it:sens2}, follows from the first as shown in~\cite[Thm.~4.1]{ABGM17}, by the same argument appearing there.
\qed

\subsection{Noise stability: proof of Theorem~\ref{thm:2}, Part~(\ref{it:stab})}\label{sec:stab2}
Let $\epsilon^3 n\to 0$ and $G_0,G_1$ as above. We will prove that $d_{\GHP}(\sM_j(G_0),\sM_j(G_1))\xrightarrow[]{\,\mathrm p\,} 0$ for every fixed $j$ by showing that, w.h.p., $\cC_j(G_1)$ is obtained from  $\cC_j(G_0)$ by gluing to it components of size $o(n^{2/3})$ and diameter $o(n^{1/3})$, which are negligible in the scaling limit.

\begin{lemma}\label{lem:Aj}
Let $j\ge 1$ be an integer, and let $A_j$ be the event  that  the component of $G_1$ containing $\cC_j(G_0)$ is $\cC_j(G_1)$. Then, $\P(A_j)\to 1$ as $n\to\infty$.
\end{lemma}
\begin{proof}
Suppose that $A_j$ does not hold, and let $i\le j$ be the smallest index for which~$A_{i}$ does not hold. Let $k$ be the index such that $\cC_{k}(G_1) \supseteq \cC_{i}(G_0)$; then
\[ \|\bX(G_1)-\bX(G_0)\|_2^2 \geq n^{-4/3}\Delta_k^2\qquad\mbox{for}\qquad \Delta_k := |\cC_k(G_1)|-|\cC_k(G_0)|\,.\]
If $k<i$ then $|\cC_k(G_1)|\ge |\cC_k(G_0)|+|\cC_i(G_0)|$ by the minimality of $i$, so $\Delta_k \geq |\cC_i(G_0)|$.
If $k> i$ then we simply use $|\cC_k(G_1)|\ge |\cC_i(G_0)|$, whence $\Delta_k \geq |\cC_i(G_0)|-|\cC_k(G_0)|$.
So, if $D^{(n)}_{i,a} = \{ |\cC_i(G_0)| < a n^{2/3}\}$ and $D^{(n)}_{i,k,a}=\{|\cC_i(G_0)|-|\cC_{k}(G_0)| < a n^{2/3}\}$, then 
\begin{align*}
\P(A_j^c) &\le \inf_{a>0}\bigg\{ \P\left(\|\bX(G_1)-\bX(G_0)\|_2\geq a\right)
+ \sum_{i\leq j}\P( D_{i,a}^{(n)}) + \sum_{i\leq j}\sum_{k>i}\P( D_{i,k,a}^{(n)})\bigg\}\,. \end{align*} 
The first term vanishes for any fixed $a>0$ as $n\to\infty$ by Corollary~\ref{cor:stab_sizes}, and the other two vanish as $a\downarrow 0$ since $\bX(G_0)$ converges to a  continues distribution in~$\selltwo$.
\end{proof}

We need the next standard fact, which follows immediately from the work of Lyons~\cite{Lyons92} connecting percolation, random walks and electrical networks on trees. 

\begin{fact}\label{fact:prob-hit-level-k}
Let $\cT$ be a $\bin(n,\frac\mu n)$-Galton--Watson tree for $\mu>0$, and let $\hgt(\cT)$ denote its height. Then for every $k$ one has $1/a_k \leq \P(\hgt(\cT) \geq k) \leq 2/a_k$, where
\[ a_k = \begin{cases} 
 1+ (1-\frac{\mu}n)\frac{1-\mu^{-k}}{\mu-1}& \mbox{if $\mu\neq 1$}\,,\\
1+(1-\frac\mu n)k & \mbox{if $\mu=1$}\,.
\end{cases}
\]
\end{fact}
\begin{proof}
Letting $\cT_{n,k}$ denote the first levels of the complete $n$-ary tree, we assign the conductance 
$\mathfrak{c}_j = 
(\mu/n)^j/(1-\mu/n)$ to every edge in between levels $j-1$ and $j$ in this tree (regarding the level of the root to be $0$). Further denoting by $\cR_k$ the  resistance between the root and level $k$ of $\cT_{n,k}$, appealing to~\cite[Thm.~2.1]{Lyons92} shows that  $1/(\cR_k+1) \leq \P(\hgt(\cT)\geq k)\leq 2/(\cR_k+1)$. Contracting the vertices in each level yields a path from $0$ to $k$ where $\tilde{\mathfrak{c}}_j =n^j \mathfrak{c}_j = \mu^j/(1-\mu/n)$ for $1\leq j \leq k$. Since we have $\cR_k=\sum_{j=1}^k \tilde{\mathfrak c}_j^{-1}$, the required inequality is established.
\end{proof}

\begin{lemma}\label{lem:Bj}
Let $j\ge 1$, $\delta>0$ and denote by $B_{j,\delta}$ the event that
\[
\sup_{v\in\cC_j(G_1)\setminus\cC_j(G_0)}\dist_{G_1}(v,\cC_j(G_0)) \le \delta n^{1/3}.
\]
Then, with $A_j$ defined as in Lemma~\ref{lem:Aj}, we have $\P(A_j \setminus B_{j,\delta})\to 0$ as $n\to\infty$.
\end{lemma}
\begin{proof}
Denote by $r_G(u):=\max_{v}\dist_G(u,v)$ where the maximum is over every vertex~$v$ that lies in the same component in $G$ as $u$. Note that if $H$ is a subgraph of $G$ and $u$ lies in an acyclic component of $G$ then $r_G(u)\ge r_H(u).$ 

Let us define a graph $G_0\subseteq G^{(j)}\subseteq G_1$ obtained by removing from $G_1$ all the edges that do not belong to $G_0$ and have an endpoint in $\cC_j(G_0)$. As such, $G_1$ is obtained from $G^{(j)}$ by connecting every pair of non-adjacent vertices in $G_0$ that has an endpoint in $\cC_j(G_0)$ independently with probability $p_1$. Therefore, $A_j\setminus B_{j,\delta}$ can hold only if an edge is added in this way to $G^{(j)}$ between $\cC_j(G_0)$ and $T$, where \[T=\left\{u\notin\cC_j(G_0)\,:\;r_{G^{(j)}}(u)>\delta n^{1/3}\right\}\]
(a shortest path from $\cC_j(G_0)$ to $v$ in $G_1$ starts with such an edge, and all its other edges belong to $G^{(j)}$).
Let $E_\kappa=\{|\cC_j(G_0)|\leq \kappa n^{2/3}\}$ for some large constant $\kappa$; writing $\P(E_\kappa)=1-\delta_\kappa$, as mentioned above one has $\delta_\kappa\to 0$ as $\kappa\to\infty$. Hence,
\begin{align}\nonumber 
    \P(A_j\setminus B_{j,\delta}) &\leq \delta_\kappa + \P(A_j\setminus B_{j,\delta}\,, E_\kappa)  \leq\delta_\kappa + 
    \E\big[ \E\left[p_1|\cC_j(G_0)||T| \one_{E_\kappa} \mid G_0\right]\big]  \\
&\leq \delta_\kappa + p_1 \kappa n^{2/3}  \E |T| \,.
\label{eq:Aj-minus-Bj}
\end{align}
Every $u\in T$ either belongs to a cyclic component of $G_1$, or has $r_{G_1}(u)>\delta n^{1/3}$. 
The expected number of vertices in cyclic components is $O(n^{2/3})$ (see~\cite[Lemma~2.2]{LPW94}), and by Fact~\ref{fact:prob-hit-level-k}, the probability that $r_{G_1}(u)>\delta n^{1/3}$ is at most $(1/\delta+o(1))n^{-1/3}$ if $\lambda=0$, and $(\lambda/(1-e^{-\lambda\delta})+o(1))n^{-1/3}$ if $\lambda\neq 0$. Thus, $\E|T|= O(n^{2/3}/\delta)$.

Combining this, as well as $p_1=(1+o(1))\epsilon/n$, with~\eqref{eq:Aj-minus-Bj}, we may now conclude that $
\P(A_j\setminus B_{j,\delta}) \le \delta_\kappa+ O(\kappa\delta^{-1}\epsilon n^{1/3}) =\delta_\kappa+o(1)$,
 since $\epsilon n^{1/3}\to 0$.
\end{proof}

\begin{lemma}\label{lem:Cj}
Let $j\ge 1$, and denote by $C_j$ the event that  $\dist_{G_0}(u,v)=\dist_{G_1}(u,v)$ for every two vertices $u,v \in \cC_j(G_0)$. Then, $\P(C_j)\to 1$.
\end{lemma}
\begin{proof}
First we claim that w.h.p.\ no two vertices in $\cC_j(G_0)$ that are non-adjacent in $G_0$ are adjacent in $G_1$.
    Indeed, the expected number of such pairs of vertices is at most $\E[|\cC_j(G_0)|^2]p_1=O(n^{1/3}\epsilon)\to 0$ as $n\to\infty$. 

Next, consider the random graph  $G_0\subseteq G^{(j)}\subseteq G_1$ that was defined in Lemma~\ref{lem:Bj}. We claim that w.h.p.\ no component of $G^{(j)}$ is connected to $\cC_j(G_0)$ in $G_1$ by more than one edge. Together with the first claim, this would imply that the shortest path in~$G_1$ between any pair of vertices in $\cC_j(G_0)$ contains only edges from $G_0$.

Let $X$ be the number of pairs of vertices that lie in the same component of $G^{(j)}$ and are both connected by an edge to $\cC_j(G_0)$ in $G_1$. For $E_\kappa=\{|\cC_j(G_0)|\le \kappa n^{2/3}\}$ as in Lemma~\ref{lem:Bj}, satisfying $\P(E_\kappa^c)=\delta_k\to 0$ as $\kappa\to\infty$, we have
\begin{align*}
\P(X>0) &\leq  \P(E_\kappa^c)+\E[X \one_{E_\kappa}] \leq \delta_k + \E\Big[ \E\big[p_1^2 |\cC_j(G_0)|^2\one_{E_\kappa}\sum_i|\cC_i(G^{(j)})|^2  \given G_0\big]\Big]\\ 
&\leq \delta_k +(1+o(1))\big(\kappa\epsilon n^{1/3}\big)^2\,\E\Big[n^{-4/3}\sum_i|\cC_i(G^{(j)})|^2\Big]\,.
\end{align*}
By Lemma \ref{lem:Aldous-G-Gbar}, $\sum_i|\cC_i(G^{(j)})|^2 \le 
\sum_i|\cC_i(G_1)|^2$ deterministically, since $G^{(j)}\subseteq G_1$. In addition, we have \[\E\Big[n^{-4/3}\sum_i|\cC_i(G_1)|^2\Big]=O(1)
\] by the convergence of $\bX(G_1)$ in $(\selltwo,\|\cdot\|_2)$. Therefore,
\[
\P(X>0) \le \delta_k + O((\kappa\epsilon n^{1/3})^2) =\delta_k +o(1)
\]
since $\epsilon n^{1/3}\to 0$.
\end{proof}
We may now conclude the proof of Theorem~\ref{thm:2}, Part~\eqref{it:stab2}.
Fix $j\ge 1$, and condition on the events $A_j\cap B_{j,\delta}\cap C_j$ as per Lemmas~\ref{lem:Aj},~\ref{lem:Bj} and~\ref{lem:Cj}. On these events, we have that $\cC_j(G_0)$ is  isometrically embedded in $\cC_j(G_1)$, and $\cC_j(G_1)$ is contained in a ball of radius $o(n^{1/3})$ about $\cC_j(G_0)$. In addition, both graph are equipped with the counting measure multiplied by $n^{-2/3}$ and they differ by only $o(n^{2/3})$ vertices. Therefore, $d_{\GHP}(\sM_j(G_0),\sM_j(G_1))=o(1)$.

The extension of convergence in the product topology to convergence in $d_{\GHP}^4$ follows, as it did in Part~\eqref{it:sens2}, from the proof of \cite[Thm.~4.1]{ABGM17}.
\qed

\subsection*{Acknowledgment} E.L.~was supported in part by NSF grant DMS-1812095.

\bibliographystyle{abbrv}
\bibliography{noise_gnp}

\end{document}